\documentclass{amsart}
\usepackage{graphicx}
\input xy
\xyoption{all}
\newenvironment{theorem}[2][Theorem]{\begin{trivlist}
\item[\hskip \labelsep {\bfseries #1}\hskip \labelsep {\bfseries #2}]}{\end{trivlist}}

\newenvironment{question}[2][Question]{\begin{trivlist}
\item[\hskip \labelsep {\bfseries #1}\hskip \labelsep {\bfseries #2}]}{\end{trivlist}}

\newtheorem{thm}{Theorem}

\newtheorem{lem}[thm]{Lemma}
\newtheorem{prop}[thm]{Proposition}
\theoremstyle{definition}
\newtheorem{defn}[thm]{Definition}
\theoremstyle{definition}

\theoremstyle{remark}

\theoremstyle{remark}

\theoremstyle{definition}
\newtheorem{ex}[thm]{Examples}
\numberwithin{equation}{section}

\begin{document}

\title[Test modules for flat dimension]{On test modules for flat dimension}%
\author{Javier Majadas}%
\address{Departamento de \'Algebra, Facultad de Matem\'aticas, Universidad de Santiago de Compostela, E15782 Santiago de Compostela, Spain}%
\email{j.majadas@usc.es}%


\begin{abstract}
  A ring with a test module of finite upper complete intersection dimension is complete intersection.
\end{abstract}
\maketitle

\section{Introduction}

Let  $(A,\mathfrak{m},k)$ be a noetherian local ring. We say that an $A$-module of finite type $M$ is a \textit{test module} if for any local homomorphism of noetherian local rings $A \rightarrow B$, any $B$-module of finite type $N$ such that $Tor_i^A(M,N) = 0$ for all $i \gg 0$ has finite flat dimension over $A$. A well known example of a test module is the residue field $k$.

When $A$ is regular, then any $A$-module of finite type is a test module. Over a hypersurface, Huneke and Wiegand \cite [Theorem 1.9]{HW} prove that test modules are precisely the ones of infinite projective dimension (alternative proofs of this fact were obtained also by Miller \cite [Theorem 1.1]{Mi} and Takahashi \cite [Corollary 7.2]{Ta2}). More generally, if $A$ is a complete intersection, Celikbas, Dao and Takahashi \cite [Proposition 2.7]{CDT} prove that test modules are those of maximal complexity. To be precise, all these characterizations were obtained with a slightly different definition of test modules, but we will show in Proposition \ref{Tt} that they also hold in our case.

Some characterizations of these properties of rings in terms of test modules also exist. First, the Auslander-Buchsbaum-Serre theorem can be stated in terms of test modules as follows:

\begin{theorem} {R.}
If (and only if) there exists a test module of finite projective dimension then $A$ is regular.
\end{theorem}

A similar criterion for Gorensteinness was obtained in \cite [Corollary 3.4]{CDT} in terms of the Gorenstein dimension of Auslander and Bridger \cite {AB} (related criteria also appear in \cite {TTY}):

\begin{theorem} {G.}
Assume that $A$ has a dualizing complex. If (and only if) there exists a test module of finite Gorenstein dimension, then $A$ is Gorenstein.
\end{theorem}

The same authors ask for a similar criterion for complete intersection rings \cite [Question 3.5]{CDT} in terms of the complete intersection dimension of \cite {AGP}:

\begin{question} {CI.}
If (and only if) there exists a test module of finite complete intersection dimension, must $A$ be a complete intersection?
\end{question}

Avramov \cite {Av} introduces the virtual projective dimension as follows (with a slight modification when the residue field of $A$ is finite). Let $A \rightarrow \hat{A}$ be the completion homomorphism (for the topology of the maximal ideal). An $A$-module of finite type $M$ is said of \textit{finite virtual projective dimension} if there exists  a deformation $Q \rightarrow \hat{A}$ (that is, a surjective homomorphism of noetherian local rings with kernel generated by a regular sequence) such that the projective dimension pd$_Q(\hat{A} \otimes_A M)$ is finite. A noetherian local ring $A$ is complete intersection if and only if its residue field is of finite virtual projective dimension if and only if any module of finite type is of finite virtual projective dimension.

Subsequently a modification of this concept, the \textit{complete intersection dimension}, was introduced in \cite {AGP}. The definition is similar, but instead of the completion homomorphism $A \rightarrow \hat{A}$, any flat local homomorphism of noetherian local rings $A \rightarrow A'$ is allowed. Complete intersection dimension shares many of the good properties with virtual projective dimension. In particular, a noetherian local ring $A$ is complete intersection if and only if its residue field has finite complete intersection dimension if and only if any module of finite type has finite complete intersection dimension. Moreover it behaves well with respect to localization. An intermediate definition, the \textit{upper complete intersection dimension}, was given in \cite {Ta1}. It allows only flat local homomorphism of noetherian local rings $(A,\mathfrak{m},k) \rightarrow (A',\mathfrak{m}',k')$ with regular closed fiber $A' \otimes_A k$.

The difficulty of the above Question CI, is that the property of being a test module does not ascend by local flat extensions. For instance, if $(A,\mathfrak{m},k) \rightarrow (A',\mathfrak{m}',k')$ is a flat local homomorphism of noetherian local rings with $A$ regular and $A'$ singular, and $M$ is a test module for $A$, then $A' \otimes_A M$ is never a test module for $A'$, since $Tor_i^{A'}(A' \otimes_A M,k') = Tor_i^A(M,k') = 0$ for all $i \gg 0$. However, we circumvent this difficulty by switching to upper complete intersection dimension, giving in this context an affirmative answer to Question CI (Theorem \ref{ci}). There are no surprises in the proof. On the contrary, the crucial step is in the definition of test module. Since the involved functor is $Tor$, we modify a little the definition of test module used in \cite {CDT} allowing modules to test not only for projectivity (of finite type modules) but also for flatness (of modules which are finite over a local homomorphism). Nevertheless, in Proposition \ref{Tt} we prove that our definition is equivalent to the one used in \cite {CDT} at least over complete intersection rings.

In the last section, we see how this modification in the definition of test module, allows us also to remove easily the supplementary hypothesis that the ring has a dualizing complex in Theorem G.

\section{Complete intersection}

\begin{defn}
A local homomorphism of noetherian local rings $(A,\mathfrak{m},k) \rightarrow (A',\mathfrak{m}',k')$ is \textit{weakly regular} if it is flat and the closed fiber $A' \otimes_A k$ is a regular local ring.  Let $f:A \rightarrow B$ be a local homomorphism of noetherian local rings. A \textit{Cohen factorization} of $f$ is a factorization $A\rightarrow A' \rightarrow B$ such that $A\rightarrow A'$ is weakly regular and $A' \rightarrow B$ is surjective. If $B$ is complete, a Cohen factorization of $f$ always exists \cite{AFH}.

A surjective homomorphism of noetherian local rings with kernel generated by a regular sequence is called a \textit{deformation}.
\end{defn}

\begin{defn}
Let $(A,\mathfrak{m},k)$ be a noetherian local ring. We say that an $A$-module of finite type $M$ is a \textit{test module} (for $A$) if for any local homomorphism of noetherian local rings $A \rightarrow B$ and any $B$-module of finite type $N$, $Tor_i^A(M,N) = 0$ for all $i \gg 0$ implies that the flat dimension of $N$ over $A$, fd$_A(N)$, is finite.
\end{defn}

\begin{ex}
(i) For any noetherian local ring $(A,\mathfrak{m},k)$, the residue field $k$ is a test module.\\*
(ii) If $A \rightarrow B$ is a finite local homomorphism of local rings and $B$ is regular, then the change of rings spectral sequence
$$E_{pq}^2=Tor_p^{B}(Tor_q^{A}(B,N),l) \Rightarrow Tor_{p+q}^A(N,l)$$
(where $l$ is the residue field of $B$) shows that $B$ is a test module for $A$.

\end{ex}

Note that our definition of test module is, a priori, more restrictive than the one used in \cite{CDT}: in that paper, an $A$-module of finite type $M$ is called a \textit{Test module} if for any $A$-module of finite type $N$, $Tor_i^A(M,N) = 0$ for all $i \gg 0$ implies that fd$_A(N)<\infty$. Over a complete intersection ring, we will see that Test modules and test modules are the same thing.

The complexity of an $A$-module of finite type $M$, cx$_A(M)$, is defined as the least integer $c \geq 0$ such that there exists some real number $\alpha$ verifying dim$_k(Tor^A_n(M,k)) \leq \alpha n^{c-1}$ for all $n \gg 0$ \cite [5.1] {AGP}. Over a complete intersection ring $A$, the possible complexities of $A$-modules of finite type are those values between 0 and codim$(A) :=$ edim$(A)$ - dim$(A)$ \cite [5.6, 5.7] {AGP}.

If $A$ is a complete intersection ring, a finite type $A$-module $M$ is a Test module if and only if cx$_A(M)$ = codim$(A)$. The ``if'' part is a direct consequence of a result of C. Miller \cite [Proposition 2.2]{Mi} and was also obtained in \cite [Corollary 1.2]{Jo}: if cx$_A(M)$ = codim$(A)$, and $N$ is a finite type $A$-module such that $Tor_i^A(M,N) = 0$ for all $i \gg 0$, then $Tor_i^A(M,\Omega^tN) = 0$ for all $i > 0$ for some syzygy $\Omega^tN$ of $N$. By \cite [loc.cit.]{Mi}, cx$_A(M \otimes_A \Omega^tN)$ = cx$_A(M)$ + cx$_A(\Omega^tN)$ which means, since the complexity of $M$ is maximal, that cx$_A(\Omega^tN)= 0$, that is, pd$_A(\Omega^tN)<\infty$ where pd denotes projective dimension. Thus pd$_A(N)<\infty$.

The ``only if'' part was proved in \cite [Proposition 2.7]{CDT}.

\begin{prop}\label{Tt}
Let $A$ be a complete intersection ring, and $M$ an $A$-module of finite type. Then $M$ is a test module if and only if it is a Test module.
\end{prop}

\begin{proof}
Since test $\Rightarrow$ Test $\Rightarrow$ maximal complexity, we only have to prove that if cx$_A(M)$ = codim$(A)$ then $M$ is a test module. So let $A \rightarrow B$ be a local homomorphism of noetherian local rings and $N$ a $B$-module of finite type such that $Tor_i^A(M,N) = 0$ for all $i \gg 0$. Let $A \rightarrow R \rightarrow \hat{B}$ be a Cohen factorization where $\hat{B}$ is the completion of $B$. By flat base change, $Tor_i^R(R\otimes_AM,\hat{N}) = Tor_i^A(M,\hat{N})=Tor_i^A(M,N)\otimes_B\hat{B}=0$ for all $i \gg 0$. Since $A \rightarrow R$ is weakly regular, codim$(A)$ = codim$(R)$. Also, cx$_A(M)$ = cx$_R(R\otimes_AM)$ by \cite [Proposition 5.2]{AGP}. That is, cx$_R(R\otimes_AM)$ = codim$(R)$, so as we have seen above, $R\otimes_AM$ is a Test module for $R$, and then $Tor_i^R(R\otimes_AM,\hat{N})=0$ for all $i \gg 0$ implies pd$_R(\hat{N})<\infty$. Since $A \rightarrow R$ is flat, we have fd$_A(\hat{N})<\infty$ and then fd$_A(N)<\infty$.
\end{proof}

\begin{lem}\label{ascent}
Let $(A,\mathfrak{m},k) \rightarrow (A',\mathfrak{m}',k')$ be a weakly regular homomorphism. If $M$ is a test module for $A$, then $A' \otimes_A M$ is a test module for $A'$.

\end{lem}
\begin{proof}
Let $A' \rightarrow B$ be a local homomorphism of noetherian local rings and $N$ a finite $B$-module such that $Tor_i^{A'}(A'\otimes_AM,N) = 0$ for all $i \gg 0$. Since $A'$ is $A$-flat, we have $Tor_i^A(M,N) = 0$ for all $i \gg 0$, and then by hypothesis fd$_A(N)<\infty$.

Consider the change of rings spectral sequence

$$E_{pq}^2=Tor_p^{A'\otimes_Ak}(Tor_q^{A'}(A'\otimes_Ak,N),k') \Rightarrow Tor_{p+q}^{A'}(N,k').$$

The ring $A'\otimes_Ak$ is regular and so $Tor_p^{A'\otimes_Ak}(-,-)=0$ for all $p \gg 0$. Also, $Tor_q^{A'}(A'\otimes_Ak,N) = Tor_q^A(k,N) = 0$ for all $q \gg 0$. That is $E_{pq}^2=0$ for all $p+q \gg 0$, and so $Tor_n^{A'}(N,k')=0$ for all $n \gg 0$. Therefore fd$_{A'}(N)<\infty$.

\end{proof}

\begin{defn} \cite {Ta1}
 We say that a finite module $M\neq0$ over a noetherian local ring $A$ has finite \textit{upper complete intersection dimension} and denote it by CI*-dim$_A(M)<\infty$ if there exist a weakly regular homomorphism $A \rightarrow A'$ and a deformation $Q \rightarrow A'$ such that pd$_Q(A'\otimes_AM)<\infty$.

\end{defn}

\begin{thm}\label{ci}
Let $M$ be a test $A$-module. If (and only if) CI*-dim$_A(M)<\infty$ then A is a complete intersection ring.
\end{thm}
\begin{proof}
Consider a weakly regular homomorphism $A \rightarrow A'$ and a deformation $Q \rightarrow A'$ such that pd$_Q(A'\otimes_AM)<\infty$. Since $A'\otimes_AM$ is a test module for $A'$ by Lemma \ref{ascent}, then $A'$ is a Test module for $Q$ \cite [Corollary 2.6]{CDT}. Since $n:=$ pd$_Q(A')<\infty$, $Tor_i^Q(A',N)=0$ for all $i>n$, and then pd$_Q(N) \leq n$ for any $Q$-module $N$ of finite type. Therefore $Q$ is a regular local ring. Then $A'$ is a complete intersection and by flat descent so is $A$.

\end{proof}

\section{Gorenstein}

\begin{defn}\label{CMdim}\cite {AB}
Let $A$ be a noetherian local ring and $M$ an $A$-module of finite type. We say that G-dim$_A(M)=0$ if the following three conditions hold:\\
(i) The canonical homomorphism $M \rightarrow $ Hom$_A$(Hom$_A$($M$,$A$),$A$) is an isomorphism.\\
(ii) $Ext_A^i(M,A)=0$ for all $i >0$.\\
(iii) $Ext_A^i($Hom$_A(M,A),A)=0$ for all $i >0$.

We say that G-dim$_A(M)<\infty$ if there exists an exact sequence
$$0 \rightarrow T_n \rightarrow ... \rightarrow T_0 \rightarrow M \rightarrow 0$$
where G-dim$_A(T_i)=0$ for $i=0,...,n$.

\end{defn}

\begin{thm}\label{Gorenstein}
Let $M$ be a test $A$-module. If (and only if) G-dim$_A(M)<\infty$ then A is a Gorenstein ring.
\end{thm}
\begin{proof}
Using that for a homomorphism $A \rightarrow A'$ and an $A$-module $T$, Hom$_{A'}(A'\otimes_AT,A') = $ Hom$_A(T,A')$, and that if $T$ is of finite type and $A'$ is $A$-flat we have also Hom$_A(T,A') = A'\otimes_A$Hom$_A(T,A)$, it is easy to prove that if $\hat{A}$ is the completion of $A$, then G-dim$_{\hat{A}}(\hat{M}) \leq $ G-dim$_A(M)<\infty$ (a more general result can be seen in \cite [Corollary 5.11]{Ch}).

By Lemma \ref{ascent}, $\hat{M}$ is a test $\hat{A}$-module. So from \cite[Corollary 3.4]{CDT} we deduce that $\hat{A}$ is Gorenstein, and then so is $A$.

\end{proof}


\end{document}